\documentclass[pdflatex,sn-mathphys-num]{sn-jnl}
\usepackage{amsmath,amssymb,amsthm,microtype,tikz}
\makeatletter
\renewcommand\normalsize{\@setfontsize\normalsize{11}{13.2}\abovedisplayskip 10pt plus2pt minus2pt\belowdisplayskip\abovedisplayskip\abovedisplayshortskip 0pt plus3pt\belowdisplayshortskip 6pt plus3pt minus3pt\let\@listi\@listI}
\makeatother
\normalsize
\makeatletter
\AtBeginDocument{\renewenvironment{proof}[1][\proofname]{\par\pushQED{\qed}\normalfont\normalsize\topsep6pt plus2pt\trivlist\item[\hskip\labelsep\itshape #1.]\ignorespaces}{\popQED\endtrivlist\@endpefalse}}
\makeatother

\newcommand{\conv}{\operatorname{conv}}
\newtheoremstyle{concise}{10pt plus2pt}{10pt plus2pt}{\itshape}{0pt}{\bfseries}{.}{.5em}{}
\theoremstyle{concise}
\newtheorem{theorem}{Theorem}
\newtheorem{corollary}[theorem]{Corollary}
\hypersetup{pdfauthor={Hanyue Shen and Pavel Dubovski},pdftitle={A sharp minimum for cube-apex bodies in every dimension}}
\begin{document}
\title[A sharp minimum for cube-apex bodies]{A sharp minimum for cube-apex bodies in every dimension}
\author*[1]{\fnm{Hanyue} \sur{Shen}}\email{hshen14@stevens.edu}
\author*[1]{\fnm{Pavel} \sur{Dubovski}}\email{pdubovsk@stevens.edu}
\affil[1]{\orgdiv{Department of Mathematical Sciences}, \orgname{Stevens Institute of Technology}, \orgaddress{\city{Hoboken}, \state{NJ} \postcode{07030}, \country{United States}}}
\abstract{We determine the minimum volume product of centrally symmetric polytopes obtained by adding facet apices to a cube while fixing their total outward height. Heights and tangential coordinates vary subject to a product condition on the polar-cap denominators. In every dimension at least two, the minimum has an explicit formula, attained by concentrating the height at one opposite pair of corner apices. For positive height in dimensions at least three, every minimizing array has this structure. The product condition strictly extends coordinatewise quadratic constraints. We also obtain two bounds without that condition: a four-dimensional Mahler bound for a specified height range, and a correction measuring volume generated jointly by two apices. The latter certifies a continuous family outside the preceding sufficient conditions. The proof combines known cap volumes with a sharp aggregate height estimate and a double-pyramid construction. Lean~4 verifies the actual-volume minimum, the necessary equality structure and both extensions.}
\maketitle
\vspace{-28pt}
\noindent\hspace*{.05\textwidth}\begin{minipage}{.9\textwidth}\small
\textbf{Acknowledgements.} The authors thank the Stevens Institute for Artificial Intelligence for computational resources.
\end{minipage}\par

\section{Introduction}
The symmetric Mahler conjecture asks whether $|K|\,|K^\circ|\ge4^n/n!$ for every origin-symmetric convex body $K$; see~\cite{survey}. We study an exact minimum for cubes with marked facet apices at prescribed total outward height. Classical cap volumes~\cite{lt,konig} reduce the polar estimate to a sharp comparison of the heights. Its strictness identifies the minimizing arrays. A separate geometric correction captures volume between two added pyramids.

The ordinary Mahler bound for a cube with one opposite apex pair follows from Karasev's projection theorem for Hanner polytopes~\cite[Theorem~4.4]{karasev}. Here the new conclusion is its exact comparison with all product-compatible arrays at the same total height. Local results near cubes and Hanner polytopes~\cite{nprz,kim} concern small perturbations; our constrained formula holds for every height. Reflection and rotation symmetries~\cite{bf,is}, anti-blocking structure~\cite{ass}, isocanted cubes~\cite{pc} and shadow systems~\cite{mr} provide related settings. Our constraint instead acts on individual polar-cap denominators.

\section{The minimum and its equality structure}
Let $n\ge2$, $C=[-1,1]^n$, and
\begin{equation}\label{eq:setup}
 P=\conv(C\cup\{\pm v_1,\ldots,\pm v_n\}),\quad
 (v_i)_i=1+a_i,\quad a_i\ge0,\quad |(v_i)_j|\le1\ (j\ne i),\quad A=\sum_i a_i.
\end{equation}
Write $\Phi(P)=|P|\,|P^\circ|$, where $P^\circ=\{y:x\cdot y\le1\ \forall x\in P\}$ and $|\cdot|$ is Lebesgue volume. In particular, $C^\circ=B_1^n=\{y:\sum_j|y_j|\le1\}$. The height $A$ belongs to the marked representation. Define
\begin{equation}\label{eq:product}
 d_i=\prod_{j\ne i}((1+a_i)^2-(v_i)_j^2),\qquad
 d_i\ge a_i^{n-1}\prod_{j\ne i}(2+a_i+a_j)\quad(a_i>0).
\end{equation}
We call arrays satisfying this inequality \emph{product compatible}; inactive indices impose no denominator condition.

\begin{theorem}\label{thm:minimum}
For fixed $A\ge0$, the minimum over arrays satisfying~\eqref{eq:setup}--\eqref{eq:product} is
\begin{equation}\label{eq:min}
 \min\Phi(P)=c_nF_n(A),\qquad c_n=\frac{4^n}{n!},\qquad
 F_n(A)=\left(1+\frac An\right)\left(1-\frac{A(A+1)^{n-2}}{(A+2)^{n-1}}\right).
\end{equation}
It is attained by $P_A=\conv(C\cup\{\pm(1+A,1,\ldots,1)\})$. Moreover, $F_2(A)=F_n(0)=1$, while $F_n(A)>1$ for $n\ge3$, $A>0$. In that latter case every minimizing array has exactly one active index $i$, with $a_i=A$ and $|(v_i)_j|=1$ for all $j\ne i$; its body is $P_A$ up to cube-coordinate symmetries.
\end{theorem}

\begin{proof}
The disjoint facet pyramids and the one-vertex cap formula~\cite[Section~3.3]{lt} give
\begin{equation}\label{eq:volumes}
 |P|\ge2^n(1+A/n),\qquad |P^\circ|\ge\frac{2^n}{n!}(1-R),\qquad
 R=\sum_i r_i,\quad r_i=\frac{(1+a_i)^{n-2}a_i^n}{d_i}\ (a_i>0),
\end{equation}
with $r_i=0$ at inactive indices. Indeed, an active apex cuts only the corresponding cross-polytope vertex; the polar estimate is the cap union bound. Active denominators are positive, and height-zero points lie in $C$.

Keeping the constant and linear terms of the nonnegative product expansion yields
$\prod_{j\ne i}(2+a_i+a_j)\ge(2+a_i)^{n-2}(2+A)$. Thus
\begin{equation}\label{eq:share}
 r_i\le\frac{a_i}{A+2}\left(\frac{1+a_i}{2+a_i}\right)^{n-2}
 \le\frac{a_i}{A+2}\left(\frac{1+A}{2+A}\right)^{n-2},\qquad
 R\le R_n(A):=\frac{A(A+1)^{n-2}}{(A+2)^{n-1}}<1.
\end{equation}
Multiplying~\eqref{eq:volumes} proves the bound. For $P_A$, there are exactly two opposite facet pyramids and one opposite cap pair, so both volume bounds are equalities. The array with first row $(1+A,1,\ldots,1)$ and all other rows $(1,\ldots,1)$ is admissible, proving attainment. The identities at $n=2$ and $A=0$ follow directly. For $n\ge3$, Bernoulli's strict inequality gives
$(A+2)^{n-1}>(A+n)(A+1)^{n-2}$, which is equivalent to $F_n(A)>1$ when $A>0$.

If two heights are positive and $n\ge3$, the second inequality in~\eqref{eq:share} is strict for an active index, so equality forces a single active row. Equality then also forces
$d_i=[A(A+2)]^{n-1}$. Each of its factors is at least $A(A+2)>0$; hence every factor equals this number and $|(v_i)_j|=1$. Inactive points lie in the cube, completing the classification.
\end{proof}

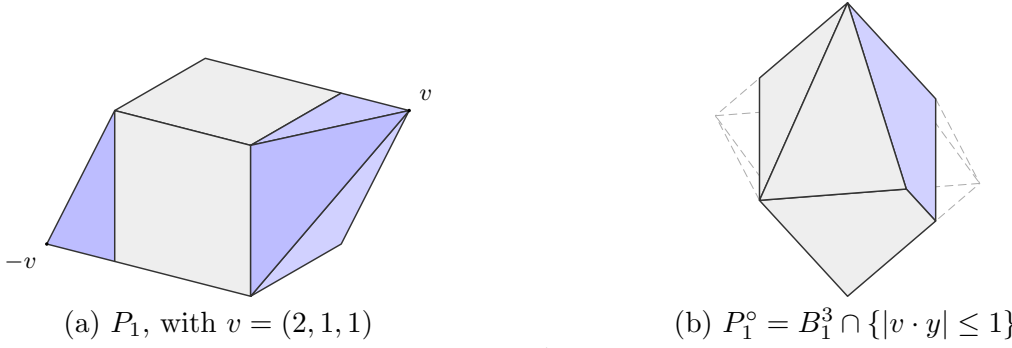
\begin{figure}[t]
\centering
\begin{minipage}[b]{.48\textwidth}\centering
\begin{tikzpicture}[line join=round,line cap=round,font=\scriptsize]
\filldraw[fill=blue!26,draw=black!80,line width=.55pt] (-2.39630,-0.88209)--(-1.49769,-1.11221)--(-1.49769,0.88209)--cycle;
\filldraw[fill=blue!18,draw=black!80,line width=.55pt] (1.49769,-0.88209)--(2.39630,0.88209)--(0.29954,-1.57243)--cycle;
\filldraw[fill=gray!13,draw=black!80,line width=.55pt] (-0.29954,1.57243)--(-1.49769,0.88209)--(0.29954,0.42187)--(1.49769,1.11221)--cycle;
\filldraw[fill=gray!13,draw=black!80,line width=.55pt] (-1.49769,0.88209)--(-1.49769,-1.11221)--(0.29954,-1.57243)--(0.29954,0.42187)--cycle;
\filldraw[fill=blue!20,draw=black!80,line width=.55pt] (0.29954,0.42187)--(2.39630,0.88209)--(1.49769,1.11221)--cycle;
\filldraw[fill=blue!26,draw=black!80,line width=.55pt] (0.29954,-1.57243)--(2.39630,0.88209)--(0.29954,0.42187)--cycle;
\fill (2.39630,0.88209) circle(.65pt) node[above right] {$v$};
\fill (-2.39630,-0.88209) circle(.65pt) node[below left] {$-v$};
\end{tikzpicture}
\par (a) $P_1$, with $v=(2,1,1)$
\end{minipage}\hfill
\begin{minipage}[b]{.48\textwidth}\centering
\begin{tikzpicture}[line join=round,line cap=round,font=\scriptsize]
\draw[gray!65,densely dashed,line width=.35pt] (1.16487,0.67116)--(1.74731,-0.44744);
\draw[gray!65,densely dashed,line width=.35pt] (1.16487,0.67116)--(-1.74731,0.44744);
\draw[gray!65,densely dashed,line width=.35pt] (-1.16487,-0.67116)--(1.74731,-0.44744);
\draw[gray!65,densely dashed,line width=.35pt] (-1.16487,-0.67116)--(-1.74731,0.44744);
\draw[gray!65,densely dashed,line width=.35pt] (0.00000,1.93890)--(1.74731,-0.44744);
\draw[gray!65,densely dashed,line width=.35pt] (0.00000,1.93890)--(-1.74731,0.44744);
\draw[gray!65,densely dashed,line width=.35pt] (0.00000,1.93890)--(1.16487,0.67116);
\draw[gray!65,densely dashed,line width=.35pt] (0.00000,1.93890)--(-1.16487,-0.67116);
\draw[gray!65,densely dashed,line width=.35pt] (0.00000,-1.93890)--(1.74731,-0.44744);
\draw[gray!65,densely dashed,line width=.35pt] (0.00000,-1.93890)--(-1.74731,0.44744);
\draw[gray!65,densely dashed,line width=.35pt] (0.00000,-1.93890)--(1.16487,0.67116);
\draw[gray!65,densely dashed,line width=.35pt] (0.00000,-1.93890)--(-1.16487,-0.67116);
\filldraw[fill=blue!18,draw=black!80,line width=.55pt] (0.77658,-0.52201)--(1.16487,-0.94459)--(1.16487,0.67116)--(0.00000,1.93890)--cycle;
\filldraw[fill=gray!13,draw=black!80,line width=.55pt] (1.16487,-0.94459)--(0.77658,-0.52201)--(-1.16487,-0.67116)--(0.00000,-1.93890)--cycle;
\filldraw[fill=gray!13,draw=black!80,line width=.55pt] (-1.16487,0.94459)--(-1.16487,-0.67116)--(0.00000,1.93890)--cycle;
\filldraw[fill=gray!13,draw=black!80,line width=.55pt] (0.00000,1.93890)--(-1.16487,-0.67116)--(0.77658,-0.52201)--cycle;
\end{tikzpicture}
\par (b) $P_1^\circ=B_1^3\cap\{|v\cdot y|\le1\}$
\end{minipage}
\caption{The attaining body and its polar in dimension three at $A=1$. Left: the cube with two added pyramids, whose exposed faces are blue. Right: the octahedron $B_1^3$, shown by dashed edges, intersected with $|2y_1+y_2+y_3|\le1$; a new polar face is blue}
\label{fig:geometry}
\end{figure}

The stronger inequalities $(v_i)_j^2\le1-a_i a_j$ imply~\eqref{eq:product} factorwise. The inclusion is strict: in dimension four take diagonal entries $3/2$, cyclic entries $(v_1)_2=(v_2)_3=(v_3)_4=(v_4)_1=1$, and all other entries zero. Then $A=2$ and $d_i=405/64>27/8$, although the stronger inequalities fail. Conversely, the facet conditions alone cannot yield~\eqref{eq:min}: $v_i=4e_i$ gives $P=4B_1^4$, $A=12$ and $\Phi(P)=32/3<c_4F_4(12)$. Figure~\ref{fig:geometry}(a,b) illustrates the attaining body and its polar.

\section{Bounds without product compatibility}
\begin{theorem}\label{thm:height}
In dimension four,~\eqref{eq:setup} and $A\le11/4$ imply $\Phi(P)\ge32/3$.
\end{theorem}
\begin{proof}
The facet bounds imply $r_i\le a_i(1+a_i)^2/(a_i+2)^3\le4a_i/27$, since
$4(a+2)^3-27(a+1)^2=(a-1)^2(4a+5)\ge0$. Hence
$R\le4A/27\le A/(A+4)$ for $A\le11/4$, and~\eqref{eq:volumes} gives the claim.
\end{proof}

For a pair of active indices $i=n-1$, $j=n$, put $a=a_i$, $b=a_j$, $\beta=1-(v_i)_j$, $\gamma=1-(v_j)_i$. Suppose $D=ab-\beta\gamma\ge0$ and define
\begin{equation}\label{eq:gain}
 p=\frac{D}{b+\beta},\qquad q=\frac{D}{a+\gamma},\qquad g=\frac{pq}{2n(n-1)}.
\end{equation}
\begin{theorem}\label{thm:pair}
Under these assumptions and~\eqref{eq:setup},
\begin{equation}\label{eq:pair}
 |P|\ge2^n(1+A/n+g),\qquad
 \Phi(P)\ge c_n(1+A/n+g)(1-R).
\end{equation}
Consequently $R\le(A/n+g)/(1+A/n+g)$ implies $\Phi(P)\ge c_n$.
\end{theorem}
\begin{proof}
The segment $[v_i,v_j]$ contains
$z_i=(bv_i+\beta v_j)/(b+\beta)$ and $z_j=(\gamma v_i+av_j)/(a+\gamma)$, whose $(i,j)$ coordinates are $(1+p,1)$ and $(1,1+q)$. Set
$F=\{x\in C:x_i=x_j=1\}$ and $W=\conv(F\cup\{z_i,z_j\})\subset P$.
Two successive pyramid formulas give $|W|=2^{n-2}pq/[n(n-1)]$, including $D=0$ and $n=2$ with zero-dimensional base volume one. Figure~\ref{fig:paircell} shows the cell and its normal projection. Since $W\subset\{x_i,x_j\ge1\}$, it meets the cube and each original facet pyramid only in coordinate hyperplanes. The same holds for $-W$, which is disjoint from $W$. Adding both cell volumes proves the primal bound. Multiply by~\eqref{eq:volumes} if $R\le1$; if $R>1$, the asserted product bound is negative.
\end{proof}

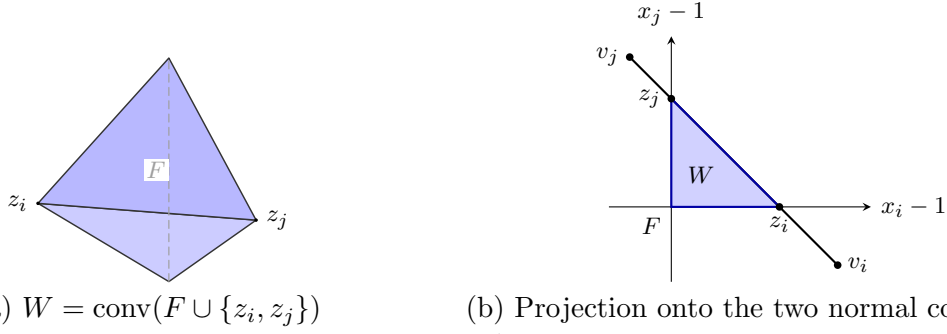
\begin{figure}[t]
\centering
\begin{minipage}[b]{.44\textwidth}\centering
\begin{tikzpicture}[line join=round,line cap=round,font=\scriptsize]
\filldraw[fill=blue!20,draw=black!80,line width=.55pt] (-1.73066,-0.44318)--(0.00000,-1.47726)--(1.15378,-0.66477)--cycle;
\filldraw[fill=blue!28,draw=black!80,line width=.55pt] (1.15378,-0.66477)--(0.00000,1.47726)--(-1.73066,-0.44318)--cycle;
\draw[densely dashed,gray!75,line width=.45pt] (0.00000,-1.47726)--(0.00000,1.47726) node[midway,left,fill=white,inner sep=1pt] {$F$};
\fill (-1.73066,-0.44318) circle(.65pt) node[left] {$z_i$};
\fill (1.15378,-0.66477) circle(.65pt) node[right] {$z_j$};
\end{tikzpicture}
\par (a) $W=\conv(F\cup\{z_i,z_j\})$
\end{minipage}\hfill
\begin{minipage}[b]{.53\textwidth}\centering
\begin{tikzpicture}[scale=1.1,>=stealth,font=\small]
\fill[blue!18] (0,0)--(1.3,0)--(0,1.3)--cycle;
\draw[->] (-.75,0)--(2.4,0) node[right] {$x_i-1$};
\draw[->] (0,-.9)--(0,2.05) node[above] {$x_j-1$};
\draw[thick] (2,-.7)--(-.5,1.8);
\draw[thick,blue!65!black] (0,0)--(1.3,0)--(0,1.3)--cycle;
\fill (2,-.7) circle(1.3pt) node[right] {$v_i$};
\fill (-.5,1.8) circle(1.3pt) node[left] {$v_j$};
\fill (1.3,0) circle(1.3pt) node[below] {$z_i$};
\fill (0,1.3) circle(1.3pt) node[left] {$z_j$};
\node[below left] at (0,0) {$F$};
\node at (.36,.39) {$W$};
\end{tikzpicture}
\par (b) Projection onto the two normal coordinates
\end{minipage}
\caption{Two views of a selected-pair cell in dimension three. Adjoining $z_i$ to the edge $F$, then $z_j$ to the resulting triangle, gives a tetrahedron. Under the normal projection, $F$ collapses to the origin, while $z_i,z_j$ are the intercepts of the apex segment. The cell lies beyond both adjacent cube facets}
\label{fig:paircell}
\end{figure}

\begin{corollary}\label{cor:family}
Let $P_t=\conv([-1,1]^4\cup\{\pm(1,1,1+t,1),\pm(1,1,1,1+t)\})$. Then $\Phi(P_t)>32/3$ for $0<t\le12/7$. For $3/2\le t\le12/7$, this certificate succeeds although product compatibility, the height condition of Theorem~\ref{thm:height}, and the uncorrected cap criterion $R\le A/(A+4)$ all fail.
\end{corollary}
\begin{proof}
Here $A=2t$, $p=q=t$, $g=t^2/24$, and $R=2t(1+t)^2/(t+2)^3$. The normalized lower bound $B(t)=(1+t/2+t^2/24)(1-R)$ satisfies
\[
 \begin{aligned}
 B(t)-1&=\frac{tH(t)}{24(t+2)^3},\\
 H(t)&=\frac{6480}{2401}+\left(\frac{12}{7}-t\right)
 \left(t^3+\frac{82}{7}t^2+\frac{1670}{49}t+\frac{9064}{343}\right)>0.
 \end{aligned}
\]
Product compatibility would require $2+2t\le2+t$. For $t\ge3/2$, $A>11/4$ and $(1+t/2)(1-R)-1=-t(t^2-2)/[2(t+2)^2]<0$. At $t=3/2$ the corrected bound is $\Phi(P_t)\ge11387/1029>32/3$.
\end{proof}

\section*{Formal verification}
Online Resource~1 contains Lean~4.32.2 sources, pinned dependencies and build records. It proves the minimum as an actual-volume \texttt{IsLeast} statement, canonical attainment, the necessary array equality structure and both extensions. Coordinate symmetries complete the written classification in Theorem~\ref{thm:minimum}. The checked dependencies use only \texttt{propext}, \texttt{Classical.choice} and \texttt{Quot.sound}. The project repository is \href{https://github.com/Henry-Hanyue-Shen/A-sharp-minimum-for-cube-apex-bodies-in-every-dimension}{\texttt{A-sharp-minimum-for-cube-apex-bodies-in-every-dimension}}.

\section*{Statements and declarations}
No dedicated grant funding was received; the Stevens Institute for Artificial Intelligence provided computational resources. The authors declare no competing interests. No empirical datasets were generated or analysed; code and verification records are supplied in Online Resource~1. Hanyue Shen led the research and writing. Pavel Dubovski provided academic supervision. OpenAI Codex assisted mathematical exploration, Lean proof development, literature checking and manuscript preparation.

\end{document}